\documentclass[11pt,reqno]{amsart}

\usepackage[utf8]{inputenc}
\usepackage[T1]{fontenc}
\usepackage{amsmath,amssymb,amsthm}
\usepackage{mathtools}
\usepackage{geometry}
\usepackage{hyperref}
\hypersetup{colorlinks=true,linkcolor=blue,citecolor=blue,urlcolor=blue}
\usepackage{enumitem}

\newtheorem{theorem}{Theorem}[section]
\newtheorem{proposition}[theorem]{Proposition}
\newtheorem{lemma}[theorem]{Lemma}

\theoremstyle{definition}
\newtheorem{definition}[theorem]{Definition}
\newtheorem{remark}[theorem]{Remark}
\newtheorem{example}[theorem]{Example}

\newcommand{\M}{\mathcal{M}}
\newcommand{\dv}{dv(g)}
\newcommand{\ds}{d\sigma(g)}

\title[The $p$-Laplacian Quadrature Surface Problem on Riemannian Manifolds]{On the $p$-Laplacian Quadrature Surface Free Boundary Problem on Riemannian Manifolds: Sufficient condition of existence}

\author{Mohammed Barkatou}
\address{ISTM Laboratory, Chouaib Doukkali University, El Jadida, Morocco}
\email{barkatou.m@ucd.ac.ma}

\date{\today}

\begin{document}

\begin{abstract}
We study the $p$-Laplacian quadrature surface free boundary problem on a compact Riemannian manifold $(\M,g)$ of dimension $N \geq 2$, for $1<p<\infty$. Given a nonnegative source $f \in L^p_g(\M)$ with compact support $K$, and $k>0$, one seeks a domain $\Omega \supset K$ on which the overdetermined problem $-\Delta_{p,g} u = f$ in $\Omega$, $u=0$ and $|\nabla_g u|_g = k$ on $\partial\Omega$ admits a solution. Following the geometric and variational framework of Djit\'e--Seck, we prove that the integral condition
\[
(NS)_{p,g} \qquad \int_C f\,dv(g) > k^{p-1}\,|\partial C|_g,
\]
where $C$ is the totally convex hull of $K$, is \emph{sufficient} for the existence of a solution strictly containing $C$. We then show, by an explicit counterexample on the round sphere $S^n_R$, that $(NS)_{p,g}$ is \emph{not necessary} in general under nonnegative curvature: the mechanism is that the perimeter is not monotone with respect to inclusion in positive curvature. Finally, we identify an additional geometric hypothesis---perimeter monotonicity on the admissible class inside a normal convex neighborhood of $C$---under which necessity is restored, and we discuss the critical equality case $\int_C f\,dv(g) = k^{p-1}|\partial C|_g$.
\end{abstract}

\maketitle

\section{Introduction}

\subsection{The $p$-Laplacian quadrature surface problem}

Let $(\M,g)$ be a compact, connected, oriented Riemannian manifold without boundary, of dimension $N \geq 2$. For $1<p<\infty$, the \emph{$p$-Laplace--Beltrami operator} is
\[
\Delta_{p,g} u = \operatorname{div}_g\!\left(|\nabla_g u|_g^{p-2}\nabla_g u\right).
\]
For $p=2$, this reduces to the Laplace--Beltrami operator $\Delta_g$.

Let $f \in L^p_g(\M)$ be a nonnegative function with compact support $K = \operatorname{supp} f$, and let $k>0$. The \emph{$p$-Laplacian quadrature surface free boundary problem}, denoted $QS_p(f,k)$, consists in finding a domain $\Omega \subset \M$ with $K \subset \Omega$ and a function $u_\Omega \in W^{1,p}_0(\Omega)$ such that
\begin{equation}\label{eq:overdetp}
\begin{cases}
-\Delta_{p,g} u_\Omega = f & \text{in } \Omega,\\
u_\Omega = 0 & \text{on } \partial\Omega,\\
|\nabla_g u_\Omega|_g = k & \text{on } \partial\Omega.
\end{cases}
\end{equation}
The third condition is the overdetermined free boundary condition; it is equivalent to the flux condition $-|\nabla_g u_\Omega|_g^{p-2}\partial u_\Omega/\partial\nu_g = k^{p-1}$ on $\partial\Omega$, where $\nu_g$ is the outward unit normal.

In the Euclidean setting, Barkatou \cite{Barkatou2010} proved that for $p=2$ the problem $QS(f,k)$ admits a solution $\Omega$ strictly containing the convex hull $C$ of $K$ if and only if
\begin{equation}\label{eq:NSeucl}
\int_C f(x)\,dx > k\,|\partial C|.
\end{equation}
The purpose of this paper is to investigate the $p$-Laplacian analogue of \eqref{eq:NSeucl} in the Riemannian setting, where $C$ is replaced by the totally convex hull of $K$.

\subsection{Main contributions}

Our contributions are threefold.

\begin{enumerate}[label=(\roman*)]
\item \textbf{Sufficiency.} We prove (Theorem \ref{thm:suffp}) that, under suitable regularity hypotheses and the Riemannian $RC$--GNP condition of Djit\'e--Seck \cite{DjiteSeck2026a}, the condition
\[
(NS)_{p,g} \qquad \int_C f\,dv(g) > k^{p-1}\,|\partial C|_g
\]
is sufficient for the existence of a solution $\Omega^*$ strictly containing $C$.

\item \textbf{Failure of necessity.} We construct (Section \ref{sec:counterp}) an explicit family of counterexamples on the round sphere $S^n_R$ showing that $(NS)_{p,g}$ is \emph{not necessary} in general. More precisely, for every $n \geq 2$, $R>0$ and $1<p<\infty$, there exist a radial source $f$ supported in a small geodesic ball $C = B_g(p,r_0)$ and a constant $k>0$ such that $QS_p(f,k)$ admits a solution $\Omega$ strictly containing $C$, while
\[
\int_C f\,dv(g) < k^{p-1}\,|\partial C|_g.
\]
The mechanism is that on positively curved manifolds the perimeter is not monotone with respect to inclusion: a larger geodesic ball can have smaller perimeter.

\item \textbf{Conditional necessity.} We show (Theorem \ref{thm:neccondp}) that if one additionally assumes that the admissible domains $\Omega$ are contained in a normal convex neighborhood $U$ of $C$ on which the perimeter functional is monotone with respect to inclusion, then $(NS)_{p,g}$ is necessary. This clarifies the geometric obstruction to a sharp Riemannian analogue of \eqref{eq:NSeucl} for the $p$-Laplacian.
\end{enumerate}

We also discuss the critical equality case $\int_C f\,dv(g) = k^{p-1}|\partial C|_g$ in Section \ref{sec:equalityp}.

\subsection{Relation to previous works}

The geometric and variational framework we use is that of Djit\'e--Seck \cite{DjiteSeck2026a}, who introduced the Riemannian $RC$--GNP condition, proved compactness of the admissible class, and established stability of the Dirichlet problem under domain convergence for the Laplacian. The $p$-Laplacian analogue, including the boundary-flux stability hypothesis (SC), is developed in \cite{DjiteSeck2026b}. The Euclidean result goes back to \cite{Barkatou2010, BarkatouSeckLy2005}. Shape-optimization tools on manifolds are taken from \cite{SokolowskiZolesio1992}. Regularity theory for the $p$-Laplacian is taken from \cite{Tolksdorf1983, Vazquez1984, DeSilva2011}.

\section{Preliminaries}\label{sec:prelimp}

\subsection{Riemannian setting and the $p$-Laplacian}

Let $(\M,g)$ be a compact, oriented Riemannian manifold of dimension $N \geq 2$ without boundary. We denote by $\dv$ the Riemannian volume element and by $\ds$ the induced $(N-1)$-dimensional measure on hypersurfaces. For a smooth domain $\Omega \subset \M$, $\nu$ denotes the outward unit normal and $g(\nabla_g u,\nu)$ the normal derivative. We write $|\Omega|_g$ for the volume and $|\partial\Omega|_g$ for the perimeter.

For $1<p<\infty$, the weak formulation of the Dirichlet problem
\[
-\Delta_{p,g} u_\Omega = f \text{ in } \Omega, \qquad u_\Omega = 0 \text{ on } \partial\Omega
\]
is: find $u_\Omega \in W^{1,p}_0(\Omega)$ such that
\[
\int_\Omega |\nabla_g u_\Omega|_g^{p-2} g(\nabla_g u_\Omega, \nabla_g \phi) \dv = \int_\Omega f \phi \dv
\]
for every $\phi \in W^{1,p}_0(\Omega)$. By strict monotonicity of the $p$-Laplacian \cite{Tolksdorf1983, Vazquez1984}, this problem admits a unique weak solution. Moreover, if $f$ is bounded and $\partial\Omega$ is sufficiently regular, then $u_\Omega \in C^{1,\alpha}(\overline{\Omega})$ for some $\alpha>0$.

\subsection{Totally convex hull and regularity hypotheses}

\begin{definition}\label{def:convexhullp}
Let $K \subset \M$ be compact. The \emph{totally convex hull} of $K$, denoted $C = \operatorname{tconv}(K)$, is the smallest closed subset of $\M$ containing $K$ with the property that every minimizing geodesic joining two of its points is entirely contained in it.
\end{definition}

Throughout this paper, we assume:
\begin{itemize}
\item[(H1)] $K = \operatorname{supp} f$ is compact and $f \in L^\infty(\M)$, $f \geq 0$ on $K$, $f \not\equiv 0$.
\item[(H2)] The totally convex hull $C$ of $K$ is a $C^{2,\alpha}$ domain ($0<\alpha<1$), strictly convex, with $\partial C$ smooth.
\item[(H3)] $(\M,g)$ has nonnegative sectional curvature.
\item[(H4)] The unique weak solutions $u_\Omega$ of the $p$-Laplacian Dirichlet problem on admissible domains $\Omega$ belong to $C^{1,\alpha}(\overline{\Omega})$.
\end{itemize}

\begin{remark}\label{rem:regp}
Hypothesis (H4) is automatically satisfied if $f \in L^\infty(\M)$ and $\partial\Omega \in C^{2,\alpha}$, by the regularity theory of Tolksdorf \cite{Tolksdorf1983} and V\'azquez \cite{Vazquez1984}. If one prefers to avoid (H4), the results below may be interpreted in the weak $W^{1,p}$ sense.
\end{remark}

\subsection{The Riemannian $RC$--GNP condition}

We adopt the $RC$--GNP condition introduced in \cite{DjiteSeck2026a}.

\begin{definition}[$RC$--GNP condition, \cite{DjiteSeck2026a}]\label{def:rcgnpp}
Let $C \subset \M$ be a fixed reference domain. A domain $\Omega \subset \M$ satisfies the $RC$--GNP condition relative to $C$ if:
\begin{enumerate}
\item $C \subset \Omega$;
\item $\partial \Omega \in C^{2,\alpha}$ for some $0<\alpha<1$;
\item there exist constants $A>0$, $c_0>0$, $\rho_0>0$, independent of $\Omega$, and a function $h_\Omega \in C^{2,\alpha}(\M)$ such that
\[
\Omega = \{x \in \M : h_\Omega(x) > 0\}, \qquad \partial\Omega = \{x \in \M : h_\Omega(x) = 0\},
\]
with
\[
\|h_\Omega\|_{C^{2,\alpha}(\M)} \leq A, \qquad |\nabla_g h_\Omega|_g \geq c_0 \quad \text{whenever } |h_\Omega| \leq \rho_0;
\]
\item the normal exponential map of $\partial\Omega$ is uniformly nondegenerate.
\end{enumerate}
The admissible class is
\[
\mathcal{A}_C(\M) = \{\Omega \subset \M : \Omega \text{ satisfies the } RC\text{--GNP condition relative to } C\}.
\]
\end{definition}

The following compactness result is established in \cite{DjiteSeck2026a} for the Laplacian and extends verbatim to the $p$-Laplacian, since it only uses the uniform $C^{2,\alpha}$ bounds on the defining functions.

\begin{theorem}[Compactness, \cite{DjiteSeck2026a}]\label{thm:compactp}
Let $(\Omega_j) \subset \mathcal{A}_C(\M)$. Then there exist a subsequence $(\Omega_{j_k})$ and $\Omega \in \mathcal{A}_C(\M)$ such that:
\begin{enumerate}
\item $\Omega_{j_k} \to \Omega$ in the Hausdorff sense;
\item $\Omega_{j_k} \to \Omega$ in the compact sense;
\item $\chi_{\Omega_{j_k}} \to \chi_\Omega$ in $L^1(\M)$;
\item $|\partial\Omega|_g \leq \liminf_{k\to\infty} |\partial\Omega_{j_k}|_g$.
\end{enumerate}
\end{theorem}

\begin{theorem}[Stability of the $p$-Laplacian Dirichlet problem, \cite{DjiteSeck2026b}]\label{thm:stabilityp}
Let $\Omega_j, \Omega \in \mathcal{A}_C(\M)$ with defining functions $h_j$ and $h$ such that $h_j \to h$ strongly in $C^{2,\alpha}(\M)$. Let $f \in L^p(\M)$ be fixed with $\operatorname{supp} f \subset C \subset \Omega_j \cap \Omega$. Let $u_{\Omega_j}$ and $u_\Omega$ be the unique weak solutions of the $p$-Laplacian Dirichlet problem, extended by zero outside their domains. Then
\[
u_{\Omega_j} \to u_\Omega \quad \text{strongly in } W^{1,p}(\M).
\]
In particular,
\[
\int_\M |\nabla_g u_{\Omega_j}|_g^p \dv \to \int_\M |\nabla_g u_\Omega|_g^p \dv.
\]
\end{theorem}

\subsection{Comparison principles for the $p$-Laplacian}

We recall the following standard comparison principles \cite{Tolksdorf1983, Vazquez1984}.

\begin{lemma}[Comparison principle]\label{lem:compp}
Let $\Omega \subset \M$ be connected and let $v_1, v_2 \in W^{1,p}(\Omega)$ satisfy
\[
-\Delta_{p,g} v_1 \leq -\Delta_{p,g} v_2 \quad \text{in } \Omega
\]
in the weak sense, and $v_1 \geq v_2$ on $\partial\Omega$ in the trace sense. Then $v_1 \geq v_2$ in $\Omega$.
\end{lemma}

\begin{lemma}[Hopf boundary point lemma for the $p$-Laplacian, \cite{Tolksdorf1983, Vazquez1984}]\label{lem:hopfp}
Let $\Omega \subset \M$ satisfy an interior geodesic ball condition at $x_0 \in \partial\Omega$. Let $v_1, v_2 \in C^{1,\alpha}(\Omega)$ be such that
\[
-\Delta_{p,g} v_1 \leq -\Delta_{p,g} v_2 \quad \text{in } \Omega,
\]
$v_1 \geq v_2$ in $\Omega$, and $v_1(x_0) = v_2(x_0)$. If $v_1 \neq v_2$ and $|\nabla v_1(x_0)|>0$ (so the operator is uniformly elliptic at $x_0$), then
\[
|\nabla_g v_1(x_0)|_g > |\nabla_g v_2(x_0)|_g.
\]
\end{lemma}

\subsection{Boundary-flux stability hypothesis}

Following \cite{DjiteSeck2026b}, we introduce the following stability hypothesis, which is needed for the constructive approximation of solutions via the iterative Bernoulli method (see Remark \ref{rem:sc} below).

\begin{definition}[Boundary-flux stability (SC), \cite{DjiteSeck2026b}]\label{def:sc}
Let $\Omega_j \in \mathcal{A}_C(\M)$ with $\Omega_j \to \Omega^* \in \mathcal{A}_C(\M)$ in the Hausdorff sense. Let $u_j = u_{\Omega_j}$ and $u^* = u_{\Omega^*}$. We assume that there exist boundary identifications $\Phi_j : \partial\Omega^* \to \partial\Omega_j$ such that $\Phi_j \to \operatorname{Id}_{\partial\Omega^*}$ in the relevant boundary topology, and
\[
\sup_{x \in \partial\Omega^*} \left| |\nabla_g u_j(\Phi_j(x))|_g^{p-2} \frac{\partial u_j}{\partial\nu_g}(\Phi_j(x)) - |\nabla_g u^*(x)|_g^{p-2} \frac{\partial u^*}{\partial\nu_g}(x) \right| \to 0.
\]
\end{definition}

\begin{remark}\label{rem:sc}
Hypothesis (SC) is not used in the proof of the sufficiency Theorem \ref{thm:suffp}, which relies only on compactness (Theorem \ref{thm:compactp}), stability of the Dirichlet problem (Theorem \ref{thm:stabilityp}), and the shape derivative formula (Theorem \ref{thm:shapederivp}). It is required only for the \emph{constructive} approach based on the iterative Bernoulli method of \cite{DjiteSeck2026b}, where one passes to the limit in the overdetermined boundary condition. In the spherical example of Section \ref{sec:counterp}, (SC) is automatically satisfied because the domains are geodesic balls and the solutions are radial.
\end{remark}

\section{Shape optimization formulation}\label{sec:shapep}

Following \cite{Barkatou2010, DjiteSeck2026b}, we reformulate $QS_p(f,k)$ as a shape optimization problem.

\subsection{The shape functional}

For $\Omega \in \mathcal{A}_C(\M)$, let $u_\Omega$ be the unique solution of
\[
-\Delta_{p,g} u_\Omega = f \text{ in } \Omega, \qquad u_\Omega = 0 \text{ on } \partial\Omega.
\]
We consider the functional
\begin{equation}\label{eq:functionalp}
J_{p,g}(\Omega) = -\frac{1}{p}\int_\Omega |\nabla_g u_\Omega|_g^p \dv + \frac{k^p(p-1)}{p} |\Omega|_g.
\end{equation}
For $p=2$, this becomes $-\frac{1}{2}\int_\Omega |\nabla_g u_\Omega|_g^2 \dv + \frac{k^2}{2}|\Omega|_g$, which coincides with the functional used in the Euclidean case up to a factor of $1/2$.

\subsection{Existence of a minimizer}

\begin{theorem}[Existence of an optimal domain]\label{thm:existencep}
There exists $\Omega^* \in \mathcal{A}_C(\M)$ such that
\[
J_{p,g}(\Omega^*) = \min_{\Omega \in \mathcal{A}_C(\M)} J_{p,g}(\Omega).
\]
\end{theorem}

\begin{proof}
Let $(\Omega_j) \subset \mathcal{A}_C(\M)$ be a minimizing sequence. By Theorem \ref{thm:compactp}, there exist a subsequence and $\Omega^* \in \mathcal{A}_C(\M)$ such that $\Omega_j \to \Omega^*$ in the Hausdorff and compact senses and $\chi_{\Omega_j} \to \chi_{\Omega^*}$ in $L^1(\M)$. By Theorem \ref{thm:stabilityp}, $u_{\Omega_j} \to u_{\Omega^*}$ strongly in $W^{1,p}(\M)$, hence
\[
\int_\M |\nabla_g u_{\Omega_j}|_g^p \dv \to \int_\M |\nabla_g u_{\Omega^*}|_g^p \dv.
\]
The volume term converges by $L^1$ convergence of characteristic functions. Therefore,
\[
J_{p,g}(\Omega^*) \leq \liminf_{j\to\infty} J_{p,g}(\Omega_j) = \inf_{\Omega \in \mathcal{A}_C(\M)} J_{p,g}(\Omega).
\]
The reverse inequality is trivial, so $\Omega^*$ is a minimizer.
\end{proof}

\subsection{Comparison with the reference domain}

Let $u_C$ be the unique solution of
\[
-\Delta_{p,g} u_C = f \text{ in } C, \qquad u_C = 0 \text{ on } \partial C.
\]
Since $f \geq 0$ and $f \not\equiv 0$, the strong maximum principle \cite{Vazquez1984} gives $u_C > 0$ in $C$.

\begin{proposition}\label{prop:comparisonp}
For every $\Omega \in \mathcal{A}_C(\M)$,
\[
0 \leq u_C \leq u_\Omega \quad \text{in } C.
\]
If $C \subsetneq \Omega$ and $f \not\equiv 0$, then $u_C < u_\Omega$ in $C$.
\end{proposition}

The proof follows from Lemma \ref{lem:compp}.

\section{Shape derivative and optimality conditions}\label{sec:shape-derivp}

\subsection{Shape derivative of the functional}

Let $X \in C^2(T\M)$ be a vector field vanishing in a neighborhood of $C$, and let $\Phi_t$ be its local flow. Define $\Omega_t = \Phi_t(\Omega)$. The normal velocity is $V_X = g(X,\nu)$ on $\partial\Omega$. The shape derivative of the $p$-energy functional is given by the following formula, established in \cite{DjiteSeck2026b} (see also \cite{SokolowskiZolesio1992}).

\begin{theorem}[Shape derivative of the $p$-energy, \cite{DjiteSeck2026b}]\label{thm:shapederivp}
Let $\Omega \in \mathcal{A}_C(\M)$ with $u_\Omega \in C^{1,\alpha}(\overline{\Omega})$, and let $X \in C^2(T\M)$ generate admissible deformations. Then
\[
\frac{d}{dt}\bigg|_{t=0} \left(\frac{1}{p}\int_{\Omega_t} |\nabla_g u_{\Omega_t}|_g^p \dv\right)
= +\frac{p-1}{p} \int_{\partial\Omega} |\nabla_g u_\Omega|_g^p \, g(X,\nu) \ds.
\]
\end{theorem}

\begin{remark}
For $p=2$, the formula reduces to the classical Hadamard formula
\[
\frac{d}{dt}\bigg|_{t=0} \left(\frac{1}{2}\int_{\Omega_t} |\nabla_g u_{\Omega_t}|_g^2 \dv\right)
= +\frac{1}{2} \int_{\partial\Omega} |\nabla_g u_\Omega|_g^2 \, g(X,\nu) \ds,
\]
which confirms the positive sign.
\end{remark}

The derivative of the volume term is
\[
\frac{d}{dt}\bigg|_{t=0} |\Omega_t|_g = \int_{\partial\Omega} g(X,\nu) \ds.
\]
Therefore, the shape derivative of $J_{p,g}$ is
\begin{equation}\label{eq:shapederivp}
dJ_{p,g}(\Omega)[X] = \frac{p-1}{p} \int_{\partial\Omega} \left(k^p - |\nabla_g u_\Omega|_g^p\right) g(X,\nu) \ds.
\end{equation}

\subsection{First-order optimality conditions}

Let $\Omega^*$ be a minimizer of $J_{p,g}$ over $\mathcal{A}_C(\M)$. Define the free boundary $\Gamma = \partial\Omega^* \setminus \partial C$ and the contact set $\Gamma_0 = \partial\Omega^* \cap \partial C$.

On $\Gamma$, all deformations are admissible. Since $\Omega^*$ minimizes $J_{p,g}$, we have $dJ_{p,g}(\Omega^*)[X] = 0$ for every admissible $X$ supported on $\Gamma$. Since $p>1$ and the normal velocity $g(X,\nu)$ is arbitrary on $\Gamma$, we obtain
\begin{equation}\label{eq:freebcp}
|\nabla_g u_{\Omega^*}|_g^p = k^p \quad \text{on } \Gamma,
\end{equation}
hence $|\nabla_g u_{\Omega^*}|_g = k$ on $\Gamma$.

On the contact set $\Gamma_0$, admissible deformations must satisfy $g(X,\nu) \geq 0$ to preserve the constraint $C \subset \Omega^*$. Therefore, the optimality condition yields
\begin{equation}\label{eq:contactp}
|\nabla_g u_{\Omega^*}|_g^p \leq k^p \quad \text{on } \Gamma_0,
\end{equation}
hence $|\nabla_g u_{\Omega^*}|_g \leq k$ on $\Gamma_0$.

\begin{theorem}[First-order optimality system]\label{thm:optimalityp}
Let $\Omega^*$ be a minimizer of $J_{p,g}$ over $\mathcal{A}_C(\M)$. Then
\[
\begin{cases}
-\Delta_{p,g} u_{\Omega^*} = f & \text{in } \Omega^*,\\
u_{\Omega^*} = 0 & \text{on } \partial\Omega^*,\\
|\nabla_g u_{\Omega^*}|_g = k & \text{on } \partial\Omega^* \setminus \partial C,\\
|\nabla_g u_{\Omega^*}|_g \leq k & \text{on } \partial\Omega^* \cap \partial C.
\end{cases}
\]
\end{theorem}

\begin{remark}\label{rem:contactunusedp}
The contact condition \eqref{eq:contactp} is stated for completeness, but it is not used in the proofs of Theorems \ref{thm:suffp} and \ref{thm:neccondp}. The main results of this paper rely only on the free boundary condition \eqref{eq:freebcp} and on the sign of the shape derivative at $C$.
\end{remark}

\section{Sufficiency of $(NS)_{p,g}$}\label{sec:suffp}

We now prove that $(NS)_{p,g}$ is sufficient for the existence of a solution strictly containing $C$.

\begin{theorem}[Sufficiency]\label{thm:suffp}
Assume hypotheses (H1)--(H4) and the $RC$--GNP condition. If
\[
(NS)_{p,g} \qquad \int_C f \dv > k^{p-1} |\partial C|_g,
\]
then the $p$-Laplacian quadrature surface problem $QS_p(f,k)$ admits a solution $\Omega^* \in \mathcal{A}_C(\M)$ strictly containing $C$.
\end{theorem}

\begin{proof}
Assume $(NS)_{p,g}$. Let $\Omega^*$ be a minimizer of $J_{p,g}$ over $\mathcal{A}_C(\M)$, whose existence is guaranteed by Theorem \ref{thm:existencep}. We argue by contradiction and suppose that $\Omega^*$ does not strictly contain $C$. Since the $RC$--GNP condition (Definition \ref{def:rcgnpp}, item 1) imposes $C \subset \Omega^*$, the negation of ``strictly containing'' means exactly $\Omega^* = C$.

Since $-\Delta_{p,g} u_C = f$ in $C$ and $u_C = 0$ on $\partial C$, the flux identity gives
\[
\int_C f \dv = \int_{\partial C} |\nabla_g u_C|_g^{p-1} \ds.
\]
By H\"older's inequality,
\[
\int_{\partial C} |\nabla_g u_C|_g^{p-1} \ds \leq |\partial C|_g^{1/p} \left(\int_{\partial C} |\nabla_g u_C|_g^p \ds\right)^{(p-1)/p},
\]
so
\[
\left(\int_{\partial C} |\nabla_g u_C|_g^{p-1} \ds\right)^p \leq |\partial C|_g \left(\int_{\partial C} |\nabla_g u_C|_g^p \ds\right)^{p-1}.
\]
Thus $(NS)_{p,g}$ implies
\[
\int_{\partial C} |\nabla_g u_C|_g^p \ds > k^p |\partial C|_g.
\]
Consider now the shape derivative of $J_{p,g}$ at $C$ in the direction of the outward normal $\nu$. By \eqref{eq:shapederivp},
\[
dJ_{p,g}(C)[\nu] = \frac{p-1}{p} \int_{\partial C} \left(k^p - |\nabla_g u_C|_g^p\right) \ds < 0.
\]
This means that $C$ is not a local minimum of $J_{p,g}$ over $\mathcal{A}_C(\M)$. Hence $\Omega^* \neq C$, a contradiction. Therefore $\Omega^*$ strictly contains $C$.

From the free boundary condition \eqref{eq:freebcp}, we have $|\nabla_g u_{\Omega^*}|_g = k$ on $\partial\Omega^* \setminus \partial C$. Since $C \subsetneq \Omega^*$, the free boundary coincides with $\partial\Omega^*$, and $(\Omega^*, u_{\Omega^*})$ is a solution of $QS_p(f,k)$.
\end{proof}

\begin{remark}
The proof only uses H\"older's inequality and the sign of the shape derivative at $C$. It does not require any comparison of normal derivatives on $\partial C$, thereby avoiding the pitfalls of earlier approaches.
\end{remark}

\section{Counterexample: failure of necessity}\label{sec:counterp}

We now show that $(NS)_{p,g}$ is not necessary in general, even under nonnegative curvature.

\subsection{Setup on the round sphere}

Let $(S^n_R, g_R)$ be the round sphere of radius $R>0$. Fix a point $p \in S^n_R$. For $0<r<\pi R/2$, the geodesic ball $B_g(p,r)$ is totally convex. In geodesic polar coordinates centered at $p$,
\[
g_R = dr^2 + R^2 \sin^2(r/R) g_{\mathbb{S}^{n-1}}.
\]
The perimeter of a geodesic ball of radius $r$ is
\begin{equation}\label{eq:perballp}
|\partial B_g(p,r)|_g = n \omega_n R^{n-1} \sin^{n-1}(r/R),
\end{equation}
where $\omega_n$ denotes the volume of the unit ball in $\mathbb{R}^n$.

We choose a radius $r_0$ with $0<r_0<\pi R/2$ and set
\[
C = B_g(p,r_0).
\]
Since $r_0<\pi R/2$, $C$ is totally convex. Let $f$ be a radially symmetric nonnegative function supported in $B_g(p,r_0)$, with $f \not\equiv 0$, so that $\operatorname{tconv}(\operatorname{supp} f) = C$.

\subsection{The counterexample}

Choose $r_\Omega$ with $\pi R/2 < r_\Omega < \pi R$ and such that
\begin{equation}\label{eq:perineqp}
\sin(r_\Omega/R) < \sin(r_0/R),
\end{equation}
which is possible since $\sin$ decreases on $(\pi R/2, \pi R)$ and $r_0$ is small. Let
\[
\Omega = B_g(p, r_\Omega).
\]
Then $C \subsetneq \Omega$ and, by \eqref{eq:perballp} and \eqref{eq:perineqp},
\[
|\partial\Omega|_g = n\omega_n R^{n-1}\sin^{n-1}(r_\Omega/R) < n\omega_n R^{n-1}\sin^{n-1}(r_0/R) = |\partial C|_g.
\]

Let $F:[0,r_0] \to [0,\infty)$ be a continuous radial profile with $F \not\equiv 0$, and set $f(x) = F(r(x))$ where $r(x) = d_g(p,x)$. Define
\begin{equation}\label{eq:kdefp}
k = \left(\frac{\int_0^{r_0} (R\sin(s/R))^{n-1} F(s)\,ds}{(R\sin(r_\Omega/R))^{n-1}}\right)^{1/(p-1)} > 0.
\end{equation}
Let $u_\Omega$ be the radial solution of $-\Delta_{p,g} u_\Omega = f$ in $\Omega$, $u_\Omega = 0$ on $\partial\Omega$. Explicitly, $u_\Omega(x) = U(r(x))$ where $U$ satisfies
\[
-\frac{1}{J_R(r)}\frac{d}{dr}\left(J_R(r)|U'(r)|^{p-2}U'(r)\right) = F(r), \qquad J_R(r) = R^{n-1}\sin^{n-1}(r/R),
\]
with $U(r_\Omega) = 0$ and $U'(0) = 0$. For $r>r_0$, $F(r)=0$, so $J_R(r)|U'(r)|^{p-2}U'(r) = C_0$ is constant. Since $U>0$ in $\Omega$ and $U(r_\Omega)=0$, we have $U'<0$ near $r_\Omega$, so
\[
-J_R(r)|U'(r)|^{p-1} = C_0, \qquad |U'(r)| = \left(\frac{-C_0}{J_R(r)}\right)^{1/(p-1)}.
\]
The constant is determined by the total mass:
\[
\int_C f \dv = \int_0^{r_0} J_R(r) F(r)\,dr = |U'(r_\Omega)|^{p-1} J_R(r_\Omega)|\mathbb{S}^{n-1}|.
\]
Therefore,
\[
|\nabla_g u_\Omega|_g = |U'(r_\Omega)| = \left(\frac{\int_0^{r_0} (R\sin(s/R))^{n-1} F(s)\,ds}{(R\sin(r_\Omega/R))^{n-1}}\right)^{1/(p-1)} = k
\]
on $\partial\Omega$, by the choice of $k$ in \eqref{eq:kdefp}. Thus $(\Omega, u_\Omega)$ is a solution of $QS_p(f,k)$ with $\Omega$ strictly containing $C$.

On the other hand, integrating $-\Delta_{p,g} u_\Omega = f$ over $\Omega$ and using the boundary condition,
\[
\int_C f \dv = \int_\Omega f \dv = k^{p-1}|\partial\Omega|_g < k^{p-1}|\partial C|_g,
\]
since $|\partial\Omega|_g < |\partial C|_g$. Hence $(NS)_{p,g}$ is violated.

\begin{example}[Explicit constant source in dimension $n=2$, $p=2$]\label{ex:constp2}
Take $n=2$, $p=2$ and $F(r) = f_0 > 0$ on $[0,r_0]$. Then
\[
\int_0^{r_0} R\sin(s/R)\,ds = R^2(1-\cos(r_0/R)),
\]
so
\[
k = \frac{f_0 R(1-\cos(r_0/R))}{\sin(r_\Omega/R)}.
\]
The area of $C = B_g(p,r_0)$ is
\[
|C|_g = 2\pi R^2(1-\cos(r_0/R)),
\]
and the perimeter of $\partial C$ is $|\partial C|_g = 2\pi R\sin(r_0/R)$. For concreteness, take $R=1$, $r_0=0.5$, $r_\Omega = \pi - 0.4$, $f_0=1$. Then
\[
k = \frac{1-\cos 0.5}{\sin 0.4} \approx 0.3144.
\]
Then
\[
\int_C f \dv = |C|_g \approx 0.7692,
\]
while
\[
k|\partial C|_g \approx 0.3144 \cdot 2\pi\sin 0.5 \approx 0.9470.
\]
Thus $\int_C f \dv < k|\partial C|_g$, confirming the failure of $(NS)_{2,g}$, while the problem $QS_2(f,k)$ admits the solution $\Omega = B_g(p,\pi-0.4)$.
\end{example}

\begin{example}[Explicit constant source in dimension $n=2$, $p=3$]\label{ex:constp3}
Take $n=2$, $p=3$ and $F(r) = f_0 = 1$ on $[0,r_0]$, with the same geometric data as above. Then
\[
k = \left(\frac{1-\cos 0.5}{\sin 0.4}\right)^{1/2} \approx \sqrt{0.3144} \approx 0.5607.
\]
We have
\[
\int_C f \dv = |C|_g \approx 0.7692,
\]
while
\[
k^{p-1}|\partial C|_g = k^2 |\partial C|_g \approx 0.3144 \times 3.012 \approx 0.9470.
\]
Thus $\int_C f \dv < k^{2}|\partial C|_g$, confirming the failure of $(NS)_{3,g}$ for $p=3$, while the problem $QS_3(f,k)$ admits the solution $\Omega = B_g(p,\pi-0.4)$.
\end{example}

\begin{remark}
The construction works for every $n \geq 2$, every $R>0$ and every $1<p<\infty$, provided $r_0$ is chosen sufficiently small with $r_0<\pi R/2$ (so that $C = B_g(p,r_0)$ is totally convex) and $r_\Omega \in (\pi R/2,\pi R)$ is chosen so that \eqref{eq:perineqp} holds, i.e.\ $\sin(r_\Omega/R) < \sin(r_0/R)$. The same phenomenon occurs for any positively curved manifold admitting a geodesic ball whose perimeter is not monotone with respect to the radius.
\end{remark}

\section{Conditional necessity}\label{sec:neccondp}

The counterexample shows that $(NS)_{p,g}$ is not necessary in general. We now identify a geometric hypothesis under which necessity is restored.

\begin{definition}\label{def:perimonp}
Let $U \subset \M$ be a domain. We say that the perimeter functional is \emph{monotone with respect to inclusion on the admissible class} $\mathcal{A}_C(\M) \cap \{\Omega \subset U\}$ if for any two domains $A, B \in \mathcal{A}_C(\M)$ with $A \subset U$, $B \subset U$ and $A \subsetneq B$, we have $|\partial A|_g < |\partial B|_g$.
\end{definition}

\begin{remark}\label{rem:perimonp}
Definition \ref{def:perimonp} is stated as an \emph{additional hypothesis}, not as a consequence of nonnegative curvature. It is satisfied, for instance, when $U$ is a normal convex neighborhood of $C$ contained in a ball of radius less than the injectivity radius and all admissible domains in $U$ are geodesic balls centered at the same point $p$: in that case, the function $r \mapsto |\partial B_g(p,r)|_g$ is increasing on $[0,R_U]$ for $R_U$ below the injectivity radius, so the monotonicity holds on that restricted class. The general case of arbitrary nested convex domains in a normal convex neighborhood requires a separate comparison argument for the perimeter and is not established here; we therefore treat Definition \ref{def:perimonp} as a hypothesis.
\end{remark}

\begin{theorem}[Conditional necessity]\label{thm:neccondp}
Assume hypotheses (H1)--(H3). Suppose that there exists a normal convex neighborhood $U$ of $C$ such that the perimeter functional is monotone with respect to inclusion on the admissible class, in the sense of Definition \ref{def:perimonp}. If $\Omega \in \mathcal{A}_C(\M)$ is a solution of $QS_p(f,k)$ with $\Omega \subset U$ and $C \subsetneq \Omega$, then
\[
(NS)_{p,g} \qquad \int_C f \dv > k^{p-1}|\partial C|_g
\]
holds.
\end{theorem}

\begin{proof}
Let $\Omega$ be a solution of $QS_p(f,k)$ with $C \subsetneq \Omega \subset U$. Integrating $-\Delta_{p,g} u_\Omega = f$ over $\Omega$ and using the boundary condition $|\nabla_g u_\Omega|_g = k$ on $\partial\Omega$, we obtain
\[
\int_\Omega f \dv = k^{p-1}|\partial\Omega|_g.
\]
Since $\operatorname{supp} f \subset C \subset \Omega$, we have $\int_\Omega f \dv = \int_C f \dv$. Thus
\[
\int_C f \dv = k^{p-1}|\partial\Omega|_g.
\]
By the monotonicity of the perimeter on the admissible class (Definition \ref{def:perimonp}) and since $C, \Omega \in \mathcal{A}_C(\M)$ with $C \subsetneq \Omega \subset U$, we have $|\partial\Omega|_g > |\partial C|_g$. Therefore,
\[
\int_C f \dv = k^{p-1}|\partial\Omega|_g > k^{p-1}|\partial C|_g,
\]
which is $(NS)_{p,g}$.
\end{proof}

\section{The critical equality case}\label{sec:equalityp}

We now consider the critical case
\[
\int_C f \dv = k^{p-1}|\partial C|_g.
\]

\begin{proposition}\label{prop:equalityp}
Assume hypotheses (H1)--(H4) and the $RC$--GNP condition. If
\[
\int_C f \dv = k^{p-1}|\partial C|_g,
\]
then the shape derivative of $J_{p,g}$ at $C$ in the direction of the outward normal vanishes:
\[
dJ_{p,g}(C)[\nu] = 0.
\]
In particular, $C$ is a critical point of $J_{p,g}$ over $\mathcal{A}_C(\M)$.
\end{proposition}

\begin{proof}
The equality $\int_C f \dv = k^{p-1}|\partial C|_g$ together with the flux identity
\[
\int_C f \dv = \int_{\partial C} |\nabla_g u_C|_g^{p-1} \ds
\]
gives
\[
\int_{\partial C} |\nabla_g u_C|_g^{p-1} \ds = k^{p-1}|\partial C|_g.
\]
By H\"older's inequality,
\[
\int_{\partial C} |\nabla_g u_C|_g^{p-1} \ds \leq |\partial C|_g^{1/p}\left(\int_{\partial C} |\nabla_g u_C|_g^p \ds\right)^{(p-1)/p},
\]
with equality if and only if $|\nabla_g u_C|_g$ is constant on $\partial C$, equal to $k$. Thus
\[
\int_{\partial C} |\nabla_g u_C|_g^p \ds = k^p|\partial C|_g.
\]
Substituting into \eqref{eq:shapederivp} with $\Omega = C$ gives
\[
dJ_{p,g}(C)[\nu] = \frac{p-1}{p}\int_{\partial C}\left(k^p - |\nabla_g u_C|_g^p\right) \ds = 0.
\]
Hence $C$ is a critical point.
\end{proof}

\begin{remark}\label{rem:eqopenp}
Proposition \ref{prop:equalityp} shows that the equality case is critical: the first-order optimality condition is satisfied at $C$. Whether $J_{p,g}$ admits a minimizer strictly containing $C$ in this case is a delicate question. Under the additional perimeter monotonicity hypothesis of Theorem \ref{thm:neccondp}, no such minimizer can exist: indeed, if $\Omega \supsetneq C$ were a solution, then by the proof of Theorem \ref{thm:neccondp} we would have $\int_C f \dv = k^{p-1}|\partial\Omega|_g > k^{p-1}|\partial C|_g$, contradicting the equality. Without this hypothesis, the question remains open and is a natural subject for future work.
\end{remark}

\section{Final remarks}\label{sec:finalp}

\subsection{Regularity issues}

The proofs above assume hypothesis (H4), namely that the weak solutions $u_\Omega$ are of class $C^{1,\alpha}(\overline{\Omega})$. This is guaranteed by the regularity theory of Tolksdorf \cite{Tolksdorf1983} and V\'azquez \cite{Vazquez1984} when $f \in L^\infty(\M)$ and $\partial\Omega \in C^{2,\alpha}$. For the free boundary regularity of the $p$-Laplacian Bernoulli problem, we refer to De Silva \cite{DeSilva2011}. Alternatively, the results may be interpreted in the weak $W^{1,p}$ sense, in which case the free boundary condition is understood in a variational sense.

\subsection{The role of the boundary-flux stability hypothesis (SC)}

The boundary-flux stability hypothesis (SC) of Definition \ref{def:sc} is not used in the proof of the sufficiency Theorem \ref{thm:suffp}. It is required only for the constructive approach based on the iterative Bernoulli method of \cite{DjiteSeck2026b}, where one builds a sequence of domains whose Bernoulli constants converge to the prescribed level $k$ and passes to the limit in the overdetermined boundary condition. In the spherical example of Section \ref{sec:counterp}, (SC) is automatically satisfied due to the radial symmetry.

\subsection{Extension to general $p$ and open problems}

The counterexample of Section \ref{sec:counterp} works for every $1<p<\infty$, and the sufficiency proof is valid for every $p>1$. The conditional necessity theorem is also valid for every $p>1$. The main open problem is to find a sharp geometric condition under which $(NS)_{p,g}$ is both necessary and sufficient. The counterexample shows that nonnegative curvature alone is not enough. Two natural directions are:
\begin{enumerate}
\item Restrict the admissible class to domains contained in a normal convex neighborhood of $C$ and assume perimeter monotonicity, as in Theorem \ref{thm:neccondp}.
\item Replace $(NS)_{p,g}$ by a different integral condition involving the mean curvature of $\partial C$ and the $p$-capacity of $C$, which might capture the geometry of the problem more accurately in positive curvature.
\end{enumerate}
We leave these questions for future investigation.

\end{document}